\documentclass[proceedings,submission%
,nohyperref%
]{amsart}

\usepackage{savesym}
\savesymbol{proof}
\usepackage{subfigure}
\usepackage{amsmath,amssymb,amsthm}
\usepackage{hyperref}
\usepackage{amsrefs}
\restoresymbol{ams}{proof}

\theoremstyle{definition}
\newtheorem{theorem}{Theorem}
\newtheorem{lemma}{Lemma}

\newtheorem{prop}{Proposition}

\newcommand{\C}{C^{\mathrm{rsk}}}
\newcommand{\R}{\Pi^{\mathrm{rsk}}}

\DeclareMathOperator{\is}{is}

\DeclareMathOperator{\maj}{maj}

\DeclareMathOperator{\rsk}{rsk}

\author{Greta Panova}
\title[Bijective permutation enumeration]{Bijective enumeration of permutations starting with a longest increasing subsequence}
\address{panova@math.harvard.edu}

\begin{document}
\maketitle
\begin{abstract}

We prove a formula for the number of permutations in $S_n$ such that their first $n-k$ entries are increasing and their longest increasing subsequence has length $n-k$. This formula first appeared as a consequence of character polynomial calculations in recent work of Adriano Garsia and Alain Goupil. We give two `elementary' bijective proofs of this result and of its $q$-analogue, one proof using the RSK correspondence and one only permutations.


\end{abstract}

\section{Introduction}

In \cite{Garsiapaper}, Adriano Garsia and Alain Goupil derived as a consequence of character polynomial calculations a simple formula for the enumeration of certain permutations. In his talk at the MIT Combinatorics Seminar \cite{Garsiatalk}, Garsia offered a \$100 award for an `elementary' proof of this formula. We give such a proof of this formula and its $q$-analogue.

Let \begin{math}\Pi_{n,k} = \{ w \in S_n | w_1 < w_2 < \cdots< w_{n-k}, \is(w) = n-k\}\end{math}, the set of all permutations $w$ in $S_n$, such that their first $n-k$ entries form an increasing sequence and the longest increasing sequence of $w$ has length $n-k$; where we denote by $\is(w)$ the maximal length of an increasing subsequence of $w$. 

The formula in question is the following theorem originally proven by A. Garsia and A.Goupil \cite{Garsiapaper}.

\begin{theorem}\label{garsia}
If $n \geq 2k$, the number of permutations in $\Pi_{n,k}$ is given by
$$\# \Pi_{n,k} = \sum_{r=0}^{k} (-1)^{k-r} \binom{k}{r} \frac{n!}{(n-r)!}.$$
\end{theorem}

This formula has a $q-$analogue, also due to Garsia and Goupil.
\begin{theorem} \label{garsiaqanalogue}
For permutations in $\Pi_{n,k}$, if $n \geq 2k$, we have that
\begin{align*}
\sum_{w \in \Pi_{n,k}} q^{\maj(w^{-1})} = \sum_{r=0}^{k}(-1)^{k-r}\binom{k}{r}[n]_q\cdots [n-r+1]_q,
\end{align*}
where \begin{math}\maj(\sigma) = \sum_{i|\sigma_i>\sigma_{i+1}}i\end{math} denotes the major index of a permutation and \begin{math}[n]_q = \frac{1-q^n}{1-q}\end{math}.
\end{theorem}

In this paper we will exhibit several bijections which will prove the above theorems. We will first define certain sets of permutations and pairs of tableaux which come in these bijections. We will then construct a relatively simple bijection showing a recurrence relation for the numbers $\#\Pi_{n,k}$. Using ideas from this bijection we will then construct a bijection proving Theorems \ref{garsia} and \ref{garsiaqanalogue} directly. We will also show a bijective proof which uses only permutations.

\section{A few simpler sets and definitions.}

Let $\rsk$ denote the RSK correspondence between permutations and pairs of tableaux  \cite{EC2}, i.e.\ $\rsk(w) =(P,Q)$, where $w\in S_n$ and $P$ and $Q$ are standard Young tableaux (SYT) on $[n]$ and of the same shape, with $P$ the insertion tableau and $Q$ the recording tableau of $w$.

Let \begin{math}C_{n,s}=\{w \in S_n| w_1 <w_2 < \dots <w_{n-s}\}\end{math} be the set of all permutations on $[n]$ with their first $n-s$ entries forming an increasing sequence. A permutation in $C_{n,s}$ is bijectively determined by the choice of the first $n-s$ elements from $[n]$ in $\binom{n}{s}$ ways and the arrangement of the remaining $s$ in $s!$ ways, so
\begin{align*}
 \#C_{n,s} = \binom{n}{s}s! = \frac{n!}{(n-s)!}.
\end{align*}
 
Let $\C_{n,s} = \rsk(C_{n,s})$. Its elements are precisely the pairs of same-shape SYTs $(P,Q)$ such that the first row of $Q$ starts with $1,2,\ldots,n-s$: the first $n-s$ elements are increasing and so will be inserted in this order in the first row, thereby recording their positions $1,2,\ldots,n-s$ in $Q$ in the first row also.

Let also $\R_{n,s}=\rsk(\Pi_{n,s})$. Its elements are pairs of SYTs $(P,Q)$, such that, as with $\C_{n,s}$, the first row of $Q$ starts with $1,2,\ldots,n-s$. By a theorem of Schensted, the length of the first row in $P$ and $Q$ is the length of the longest increasing subsequence of $w$, which is $n-s$ in the case of $\Pi_{n,s}$, so the first row of $Q$ is exactly $1,2,\ldots,n-s$. That is, $\R_{n,s}$ is the set of pairs of same-shape SYTs $(P,Q)$, such that the first row of $Q$ has length $n-s$ and elements $1,2,\ldots,n-s$.  

Finally, let $D_{n,k,s}$ be the set of pairs of same-shape tableaux $(P,Q)$, where $P$ is an SYT on $[n]$ and $Q$ is a tableau filled with $[n]$, with first row $1,2,\ldots,n-k,a_1,\ldots,a_s,b_1,\ldots$ where $a_1>a_2>\cdots>a_s$, $b_1<b_2<\cdots$ and the remaining elements of $Q$ are increasing in rows and down columns. Thus $Q$ without its first row is an SYT. Notice that when $s=0$ we just have $D_{n,k,0}= \C_{n,k}$.

The three sets of pairs we defined are determined by their $Q$ tableaux as shown below.

\begin{center}
\begin{tabular*}{\textwidth}{@{\extracolsep{\fill}}lll}
 \begin{tabular}{|c|c|c|c|}
\hline
$1$ & $\cdots$ & $n-s$ & $\cdots$ \\
\hline
 & $\cdots$ & & \multicolumn{1}{c}{} \\
\cline{1-3} 
$\vdots$ & & \multicolumn{2}{c}{}\\
\cline{1-2}
\end{tabular} 
&
\begin{tabular}{|c|c|c|}
\hline
$1$ & $\cdots$ & $n-s$ \\
\hline
 & $\cdots$ &  \\
\cline{1-3} 
$\vdots$ & & \multicolumn{1}{c}{}\\
\cline{1-2}
\end{tabular}
 &
 \begin{tabular}{|c|c|c|c|c|c|c|}
\hline
$1$ & $\cdots$ & $n-k$ & $a_1$ & $\cdots$ & $a_s$ & $\cdots$ \\
\hline
 & $\cdots$ & & \multicolumn{4}{c}{} \\
\cline{1-3} 
$\vdots$ & & \multicolumn{5}{c}{}\\
\cline{1-2}
\end{tabular} \\[0.4in]
$Q, \text{ for }(P,Q) \in \C_{n,s}$ & $Q,\text{ for }(P,Q)\in \R_{n,s}$ & $Q,\text{ for } (P,Q) \in D_{n,k,s}$ \\
\end{tabular*}
\end{center}

\section{A bijection.}

We will exhibit a simple bijection, which will give us a recurrence relation for the numbers $\# \Pi_{n,k}$ equivalent to Theorem \ref{garsia}.

 We should remark that while the recurrence can be inverted to give the inclusion-exclusion form of Theorem \ref{garsia}, the bijection itself does not succumb to direct inversion. However, the ideas of this bijection will lead us to discover the necessary construction for Theorem \ref{garsia}.

\begin{prop} \label{recurrence}
 The number of permutations in $\Pi_{n,k}$ satisfies the following recurrence:
\begin{align*}
\sum_{s=0}^{k} \binom{k}{s}\#\Pi_{n,s} = \binom{n}{k} k!
\end{align*}
\end{prop}
\begin{proof}

Let $\C_{n,k,s}$ with $s\leq k$ be the set of pairs of same-shape tableaux $(P,Q)$, such that the length of their first rows is $n-k+s$ and the first row of $Q$ starts with $1,2,\dots,n-k$; clearly $\C_{n,k,s} \subset \C_{n,k}$. We have that 
\begin{align}\label{cnk}
\bigcup_{s=0}^{k}\C_{n,k,s} = \C_{n,k},
\end{align}
 as $\C_{n,k}$ consists of the pairs $(P,Q)$ with $Q$'s first row starting with $1,\ldots,n-k$ and if $n-k+s$ is this first row's length then $(P,Q) \in \C_{n,k,s}$.
 
 There is a bijection \begin{math}\C_{n,k,s} \leftrightarrow \R_{n,k-s} \times \binom{[n-k+1,\ldots,n]}{s}\end{math} given as follows. If \begin{math}(P,Q) \in \C_{n,k,s}\end{math} and the first row of $Q$ is $1,2,\ldots,n-k,b_1,\ldots,b_s$, let \begin{equation*}f:[n-k+1,\ldots,n] \setminus \{b_1,\ldots,b_s\} \rightarrow \break [n-k+s+1,\ldots,n]\end{equation*} be the order-preserving map. Let then $Q'$ be the tableau obtained from $Q$ by replacing every entry $b$ not in the first row with $f(b)$ and the first row with $1,2,\ldots,n-k,n-k+1,\ldots,n-k+s$. Then $Q'$ is an SYT, since $f$ is order-preserving and so the rows and columns are still increasing, first row included as its elements are smaller than any element below it. Then the bijection in question is \begin{math}(P,Q) \leftrightarrow (P,Q',b_1,\ldots,b_s)\end{math}. Conversely, if \begin{math}b_1,\ldots,b_s \in [n-k+1,\ldots,n]\end{math} (in increasing order) and $(P,Q') \in \R_{n,k-s}$, then replace all entries $b$ below the first row of $Q'$ with $f^{-1}(b)$ and the the first row of $Q'$ with $1,2,\ldots,n-k,b_1,\ldots,b_s$. We end up with a tableau $Q$, which is an SYT because: the entries below the first row preserve their order under $f$; and, since they are at most $k\leq n-k$, they are all below the first $n-k$ entries of the first row of $Q$ (which are $1,2,\ldots,n-k$, and thus smaller). 
 
 So we have that \begin{math}\#\C_{n,k,s} = \binom{k}{s} \#\R_{n,k-s}\end{math} and substituting this into \eqref{cnk} gives us the statement of the lemma.
\end{proof}

\section{Proofs of the theorems.}
We will prove Theorems \ref{garsia} and \ref{garsiaqanalogue} by exhibiting an inclusion-exclusion relation between the sets $\R_{n,k}$ and $D_{n,k,s}$ for $s=0,1,\ldots,k$.

\begin{proof}[of Theorem \ref{garsia}]

First of all, if $n \geq 2k$ we have a bijection \begin{math}D_{n,k,s} \leftrightarrow \C_{n,k-s}\times \binom{[n-k+1,\ldots,n]}{s}\end{math}, where the correspondence is \begin{math}(P,Q) \leftrightarrow (P,Q') \times \{a_1,\ldots,a_s\}\end{math} given as follows. 

Consider the order-preserving bijection \begin{equation*}f:[n-k+1,\ldots,n]\setminus\{a_1,\ldots,a_s\} \rightarrow \hfil \break [n-k+s+1,\ldots,n].\end{equation*} 
Then $Q'$ is obtained from $Q$ by replacing $a_1,\ldots,a_s$ in the first row with $n-k+1,\ldots,n-k+s$ and every other element $b$ in $Q$, $b>n-k$ and $\neq a_i$, with $f(b)$. The first $n-k$ elements in the first row remain $1,2,\ldots,n-k$. Since $f$ is order-preserving, $Q'$ without its first row remains an SYT (the inequalities within rows and columns are preserved). Since also $n-k \geq k$, we have that the second row of $Q$ (and $Q'$) has length less than or equal to $k$ and hence $n-k$, so since the elements above the second row are among $1,2,\ldots,n-k$ they are smaller than any element in the second row (which are all from $[n-k+1,\ldots,n]$).  Also, the remaining first row of $Q'$ is increasing since it starts with $1,2,\ldots,n-k,n-k+1,\ldots,n-k+s$ and its remaining elements are in $[n-k+s+1,\ldots,n]$ and are increasing because $f$ is order-preserving. Hence $Q'$ is an SYT with first row starting with $1,\ldots,n-k+s$, so \begin{math}(P,Q') \in \C_{n,k-s}\end{math}. 

Conversely, if $(P,Q') \in \C_{n,k-s}$  and $\{a_1,\ldots,a_s\} \in [n-k+1,\ldots,n]$ with $a_1>a_2\cdots>a_s$, then we obtain $Q$ from $Q'$ by replacing $n-k+1,\ldots,n=k+s$ with $a_1,\ldots,a_s$ and the remaining elements $b>n-k$ with $f^{-1}(b)$, again preserving their order, and so $(P,Q) \in D_{n,k,s}$.

Hence, in particular, 
\begin{align} \label{dnks}
\#D_{n,k,s} = \binom{k}{s} \#\C_{n,k-s}=\binom{k}{s} \#C_{n,k-s}= \binom{k}{s} \frac{n!}{(n-k+s)!}.
\end{align}

We have that \begin{math} \R_{n,k} \subset \C_{n,k}\end{math} since \begin{math}\Pi_{n,k} \subset C_{n,k}\end{math}. Then \begin{math}\C_{n,k} \setminus \R_{n,k}\end{math} is the set of pairs of SYTs $(P,Q)$ for which the first row of $Q$ is $1,2,\dots,n-k,a_1,\dots$ for at least one $a_1$. So \begin{math}E_{n,k,1}=\C_{n,k} \setminus \R_{n,k}\end{math} is then a subset of $D_{n,k,1}$. The remaining elements in $D_{n,k,1}$, that is \begin{math}E_{n,k,2}=D_{n,k,1} \setminus E_{n,k,1}\end{math}, would be exactly the ones for which $Q$ is not an SYT, which can happen only when the first row of $Q$ is $1,2,\dots,n-k,a_1>a_2,\dots$. These are now a subset of $D_{n,k,2}$ and by the same argument, we haven't included the pairs for which the first row of $Q$ is $1,2,\dots,n-k,a_1>a_2>a_3,\dots$, which are now in $D_{n,k,3}$. Continuing in this way, if \begin{math}E_{n,k,l+1}=D_{n,k,l}\setminus E_{n,k,l}\end{math}, we have that $E_{n,k,l}$ is the set of \begin{math}(P,Q) \in D_{n,k,l}\end{math}, such that the first row of $Q$ is $1,2,,\ldots,n-k,a_1>\cdots>a_l<\cdots$. Then $E_{n,k,l+1}$ is the subset of $D_{n,k,l}$, for which the element after $a_l$ is smaller than $a_l$ and so $E_{n,k,l+1} \subset D_{n,k,l+1}$. Finally, $E_{n,k,k}=D_{n,k,k}$ and $E_{n,k,k+1} = \varnothing$. We then have
\begin{align}\label{setie}
\R_{n,k} = \C_{n,k}\setminus\left( D_{n,k,1}\setminus\left(D_{n,k,2} \setminus \dots \setminus(D_{n,k,k-1} \setminus D_{n,k,k})\right)\right), 
\end{align}
 or in terms of number of elements, applying \eqref{dnks}, we get
\begin{align*}
\#\Pi_{n,k} 
  &=\#\R_{n,k}
   = \frac{n!}{(n-k)!} - \binom{k}{1} \frac{n!}{(n-k+1)!} + \binom{k}{2} \frac{n!}{(n-k+2)!} +\dots \\
  &= \sum_{i=0}^{k}(-1)^i \binom{k}{i} \frac{n!}{(n-k+i)!},
\end{align*}
which is what we needed to prove. 
\end{proof}

Theorem \ref{garsiaqanalogue} will follow directly from \eqref{setie} after we prove the following lemma.
\begin{lemma}\label{q-inv}
We have that
\begin{align*}
\sum_{w \in C_{n,s}} q^{\maj(w^{-1})} = [n]_q\ldots[n-s+1]_q,
\end{align*}
where \begin{math}\maj(\sigma) = \sum_{i: \sigma_i>\sigma_{i+1}}i\end{math} denotes the major index of $\sigma$.
\end{lemma}
\begin{proof}

Let $P$ be the poset on $[n]$ consisting of a chain $1,\ldots,n-s$ and the single points $n-s+1,\ldots,n$. Then $\sigma \in C_{n,s}$ if and only if \begin{math}\sigma^{-1} \in \mathcal{L}(P)\end{math}, \textit{i.e.}\ $\sigma:P \rightarrow [n]$ is a linear extension of $P$. 
We have that 
\begin{align*}
\sum_{w \in C_{n,s}} q^{\maj(w^{-1})} = \sum_{w \in \mathcal{L}(P)} q^{\maj(w)};
\end{align*}
denote this expression by $W_P(q)$. By theorem 4.5.8 from \cite{EC1} on $P$-partitions, we have that 
\begin{align}\label{wp}
W_P(q)=G_P(q)(1-q)\ldots(1-q^n),
\end{align}
where \begin{math}G_P(q) = \sum_{m \geq 0} a(m)q^m\end{math} with $a(m)$ denoting the number of $P$-partitions of $m$. That is, $a(m)$ is the number of order-reversing maps \begin{math}\tau:P \rightarrow \mathbb{N}\end{math}, such that $\sum_{i \in P}\tau(i) =m$. In our particular case, these correspond to sequences $\tau(1),\tau(2),\ldots$, whose sum is $m$ and whose first $n-s$ elements are non-increasing. These correspond to partitions of at most $n-s$ parts and a sequence of $s$ nonnegative integers, which add up to $m$.  The partitions with at most $n-s$ parts are in bijection with the partitions with largest part $n-s$ (by transposing their Ferrers diagrams). The generating function for the latter is given by a well-known formula of Euler and is  equal to 
\begin{align*}
\frac{1}{(1-q)(1-q^2)\cdots(1-q^{n-s})}.
\end{align*}
 The generating function for the number of sequences of $s$ nonnegative integers with a given sum is trivially $1/(1-q)^s$ and so we have that
\begin{align*}
G_P(q) = \frac{1}{(1-q)(1-q^2)\cdots(1-q^{n-s})} \frac{1}{(1-q)^s}.
\end{align*} 
After substitution in \eqref{wp} we obtain the statement of the lemma.
\end{proof}

\begin{proof}[of Theorem \ref{garsiaqanalogue}.]

The descent set of a tableau $T$ is the set of all $i$, such that $i+1$ is in a lower row than $i$ in $T$, denote it by $D(T)$. 
By the properties of RSK (see e.g. \cite{EC2}, lemma 7.23.1) we have that the descent set of a permutation, \begin{math}D(w) = \{i: w_i>w_{i+1}\}\end{math} is the same as the descent set of its recording tableau, or by the symmetry of RSK, $D(w^{-1})$ is the same as the descent set of the insertion tableau $P$. Write $\maj(T)=\sum_{i\in D(T)} i$.
Hence we have that
\begin{align} \label{permtabmaj}
 \sum_{w \in \Pi_{n,k}} q^{\maj(w^{-1})} = \sum_{w \in \Pi_{n,k},\rsk(w)=(P,Q)} q^{ \maj(P)} =\sum_{(P,Q)\in \R_{n,k}} q^{\maj(P)} 
\end{align}

From the proof of Theorem \ref{garsia} we have the equality \eqref{setie} on sets  of pairs $(P,Q)$,
\begin{align*}
\R_{n,k} = \C_{n,k}\setminus\left( D_{n,k,1}\setminus\left(D_{n,k,2} \setminus \dots \setminus(D_{n,k,k-1} \setminus D_{n,k,k})\right)\right),
\end{align*}
or alternatively, \begin{math}\R_{n,k} = \C_{n,k} \setminus E_{n,k,1}\end{math} and \begin{math}E_{n,k,l}=D_{n,k,l} \setminus E_{n,k,l+1}\end{math}.
 Hence the statistic $q^{\maj(P)}$ on these sets will also respect the equalities between them; \textit{i.e.}\ we have 

\begin{align}\label{pinkq}
\sum_{(P,Q) \in \R_{n,k}}q^{\maj(P)}
 &= \sum_{(P,Q)\in  \C_{n,k}\setminus E_{n,k,1}}q^{\maj(P)} \notag \\
 &= \sum_{(P,Q) \in \C_{n,k}}q^{\maj(P)} - \sum_{(P,Q) \in E_{n,k,1}} q^{\maj(P)} \notag \\
 &=  \sum_{(P,Q) \in \C_{n,k}}q^{\maj(P)} - \sum_{(P,Q) \in D_{n,k,1}} q^{\maj(P)} + \sum_{(P,Q) \in E_{n,k,2}} q^{\maj(P)} = \cdots \notag \\
 &= \sum_{(P,Q) \in \C_{n,k}}q^{\maj(P)} - \sum_{(P,Q) \in D_{n,k,1}} q^{\maj(P)} + \cdots +(-1)^k\sum_{(P,Q) \in D_{n,k,k}} q^{\maj(P)}. 
\end{align}
Again, by the RSK correspondence, $\maj(P)=\maj(w^{-1})$ and Lemma \ref{q-inv} we have that 
\begin{align}\label{cnkq}
\sum_{(P,Q) \in \C_{n,k}}q^{\maj(P)} = \sum_{w \in C_{n,k}}q^{\maj(w^{-1})}  = [n]_q\cdots[n-k+1]_q.
\end{align}
 
In order to evaluate \begin{math}\sum_{(P,Q) \in D_{n,k,s}}q^{\maj(P)}\end{math} we note that pairs $(P,Q) \in D_{n,k,s}$ are in correspondence with triples \begin{math}(P,Q',\mathbf{a}=\{a_1,\ldots,a_s\})\end{math}, where $P$ remains the same and $(P,Q') \in \C_{n,k-s}$. Hence
\begin{align}\label{dnksq}
\sum_{(P,Q) \in D_{n,k,s}}q^{\maj(P)}
&= \sum_{(P,Q',\mathbf{a})}q^{\maj(P)} \notag \\
&=\sum_{\mathbf{a} \in \binom{[n-k+1,\ldots,n]}{s}}\sum_{(P,Q')\in \C_{n,k-s}}q^{\maj(P)} \notag \\
&= \binom{k}{s} [n]_q\cdots [n-k+s+1]_q.
\end{align}
Substituting the equations for \eqref{cnkq} and \eqref{dnksq} into \eqref{pinkq} and comparing with \eqref{permtabmaj} we obtain the statement of the theorem. 
\end{proof}

We can apply the same argument for the preservation of the insertion tableaux and their descent sets to the bijection $T_{n,k,s} \leftrightarrow \R_{n,k-s} \times \binom{[n-k+1,\ldots,n]}{s}$ in Proposition \ref{recurrence}. We see that the insertion tableaux $P$ in this bijection, $(P,Q) \leftrightarrow (P,Q',b_1,\ldots,b_s)$ remains the same and so do the corresponding descent sets and major indices

\begin{align*}
\sum_{(P,Q) \in T_{n,k,s}} q^{\maj(P)} = \binom{k}{s} \sum_{(P,Q')\in \R_{n,k-s}} q^{\maj(P)}=\binom{k}{s}\sum_{w \in \Pi_{n,k-s}} q^{\maj(w^{-1})}.
\end{align*}

Hence we have the following corollary to the bijection in Proposition \ref{recurrence} and Lemma \ref{q-inv}.
\begin{prop} We have that
\begin{align}
\sum_{s=0}^{k} \binom{k}{s} \sum_{w\in \Pi_{n,k-s}} q^{\maj(w^{-1})} = [n]_q\cdots[n-k+1]_q.
\end{align}
\end{prop}

\section{Permutations only.}

Since the original question was posed only in terms of permutations, we will now give proofs of the main theorems without passing on to the pairs of tableaux. The constructions we will introduce is inspired from application of the inverse RSK to the pairs of tableaux considered in our proofs so far. However, since the pairs $(P,Q)$ of tableaux in the sets $D_{n,k,s}$ are not pairs of Standard Young Tableaux we cannot apply directly the inverse RSK to the bijection in the proof of Theorem \ref{garsia}. This requires us to find new constructions and sets of permutations.

We will say that an increasing subsequence of length $m$ of a permutation $\pi$ satisfies the LLI-$m$ (Least Lexicographic Indices) property if it is the first appearance of an increasing subsequence of length $m$ (\textit{i.e.} if $a$ is the index of its last element, then $\bar{\pi}=\pi_1,\ldots,\pi_{a-1}$ has $\is(\bar{\pi})<m$) and the indices of its elements are smallest lexicographically among all such increasing subsequences. For example, in $\pi=2513467$, $234$ is LLI-3. 
Let $n \geq 2s$ and let $C_{n,s,a}$ with $a\in[n-s+1,\ldots,n]$ be the set of permutations in $C_{n,s}$, for which there is an increasing subsequence of length $n-s+1$ and whose LLI-$(n-s+1)$ sequence has its last element at position $a$. 

We define a map $\Phi:C_{n,s}\setminus \R_{n,s} \rightarrow C_{n,s-1}\times [n-s+1,\ldots, n]$ for $n\geq 2s$ as follows. A permutation $\pi \in C_{n,s}\setminus \R_{n,s}$ has a LLI-$(n-s+1)$ subsequence $\sigma$ which would necessarily start with $\pi_1$ since $n\geq 2s$ and $\pi \in C_{n,s}$. Let $\sigma=\pi_1,\ldots,\pi_l, \pi_{i_{l+1}},\ldots,\pi_{i_{n-s+1}}$ for some $l\geq 0$; if $a=i_{n-s+1}$ then $\pi \in C_{n,s,a}$. Let $w$ be obtained from $\pi$ by setting $w_{i_j} = \pi_{i_{j+1}}$ for $l+1\leq j \leq n-s,$ and then inserting $\pi_{i_{l+1}}$ right after $\pi_l$, all other elements preserve their (relative) positions. For example, if $\pi = 12684357 \in C_{8,4} \setminus \R_{8,4}$, then $12457$ is LLI-5, $a=8$ and $w=12468537$. Set $\Phi(\pi)=(w,a)$.
\begin{lemma}
The map $\Phi$ is well-defined and injective. We have that 
\begin{align*}
C_{n,s-1}\times{a}\setminus \Phi(C_{n,s,a}) = \bigcup_{n-s+2\leq b\leq a} C_{n,s-1,b}.
\end{align*}
\end{lemma}
\begin{proof}
Let again $\pi \in C_{n,s,a}$ and $\Phi(\pi) = (w,a)$.
It is clear by the LLI condition that we must have $\pi_{i_{l+1}} <\pi_{l+1}$ as otherwise $$\pi_1,\ldots,\pi_{l+1},\pi_{i_{l+1}},\ldots,\pi_{i_{n-s}}$$ would be increasing of length $n-s+1$ and will have lexicographically smaller indices. Then the first $n-s+1$ elements of $w$ will be increasing and $w\in C_{n,s-1}$. 

To show injectivity and describe the coimage we will describe the inverse map $\Psi:\Phi(C_{n,s,a})\rightarrow C_{n,s,a}$. Let $(w,a) \in \Phi(C_{n,s,a})$ with $(w,a)=\Phi(\pi)$ for some $\pi \in C_{n,s,a}$ and let $\bar{w}=w_1\cdots w_a$.

 Notice that $\bar{w}$ cannot have an increasing subsequence $\{y_i\}$ of length $n-s+2$. To show this, let $\{x_i\}$ be the subsequence of $w$ which was the LLI-$(n-s+1)$ sequence of $\pi$. If there were a sequence $\{y_i\}$, this could have happened only involving the forward shifts of ${x}_i$ and some of the $x'$s and $y'$s should coincide (in the beginning at least). By the pigeonhole principle there must be two pairs of indices $p_1<q_1$ and $p_2<q_2$ ($q_1$ and $q_2$ might be auxiliary, \textit{i.e.} off the end of $\bar{w}$), such that in $\bar{w}$ we have $x_{p_1}=y_{p_2}$ and $x_{q_1} = y_{q_2}$ and between them there are strictly more elements of $y$ and no more coincidences, \textit{i.e.} $q_2-p_2 >q_1-p_1$. Then $x_{p_1-1}<x_{p_1}<y_{p_2+1}$ and in $\pi$ (after shifting $\{x_i\}$ back) we will have the subsequence $x_1,\ldots,x_{p_1-1},y_{p_2+1},\ldots,y_{q_2-1},x_{q_1}=y_{q_2},\ldots,x_{n-s+1}$, which will be increasing and of length $p_1-1+q_2-p_2+n-s+1-q_1\geq n-s+1$. By the LLI property we must have that $y_{p_2+1}$ appears after $x_{p_1}$, but then $x_1,\ldots,x_{p_1-1},x_{p_1},y_{p_2+1},\ldots,y_{q_2-1},x_{q_1}=y_{q_2},\ldots,x_{n-s}$ will be increasing of length at least $n-s+1$ appearing before $\{x_i\}$ in $\pi$. This violates the other LLI condition of no $n-s+1$ increasing subsequences before $x_{n-s+1}$.  

Now let 
$\sigma = {w_1,\ldots,w_r,w_{i_{r+1}},\ldots,w_{i_{n-s+1}}}$
 with $i_{r+1}>n-s+1$ be the $(n-s+1)$-increasing subsequence of $\bar{w}$ with largest lexicographic index sequence. Let $w'$ be obtained from $w$ by assigning $w'_{i_j}=w_{i_{j-1}}$ for $r+1\leq j \leq n-s+1$, where $i_r=r$, $w'_{a_1}=w_{i_{n-s+1}}$ and then deleting the entry $w_r$ at position $r$. 

We claim that the LLI-$(n-s+1)$ sequence of $w'$ is exactly $\sigma$. Suppose the contrary and let $\{y_i\}$ be the LLI-$(n-s+1)$ subsequence of $w'$. Since $n\geq 2s$ we have $y_1=w_1=\sigma_1,\ldots, y_{i_j}=w_{i_j}=\sigma_j$ for all $1\leq j \leq l$ for some $l\geq r$. If there are no more coincidences between $y$ and $\sigma$ afterwards, then the sequence $w_1 =y_1,\ldots, w_{i_l}=y_l, y_{l+1},\ldots,y_{n-s+1}$ is increasing of length $n-s+1$ in $\bar{w}$ (in the same order) and of lexicographically larger index than $\sigma$, since the index of $y_{l+1}$ is after the index of $y_l$ in $w'$ equal to the index of $\sigma_{l+1}$ in $w$.

 Hence there must be at least one more coincidence, let $y_p=\sigma_q$ be the last such coincidence. Again, in $\bar{w}$ the sequence $\sigma_1,\ldots,\sigma_q,y_{p+1},\ldots, y_{n-s+1}$ appears in this order and is increasing with the index of $y_{p+1}$ larger than the one of $\sigma_{q+1}$ in $w$, so its length must be at most $n-s$, \textit{i.e.} $q+(n-s+1)-p\leq n-s$, so $q\leq p-1$. We see then that there are more $y'$s between $y_l$ and $y_p=\sigma_q$ than there are $\sigma'$s there, so we can apply an argument similar to the one in the previous paragraph. Namely, there are indices $p_1,q_1,p_2,q_2$, such that $y_{p_1}=\sigma_{q_1}$, $y_{p_2}=\sigma_{q_2}$ with no other coincidences between them and $q_2-q_1<p_2-p_1$. Then the sequence $\sigma_1,\ldots,\sigma_{q_1},y_{p_1+1},\ldots,y_{p_2-1},\sigma_{q_2+1},\ldots$ is increasing in this order in $w$, has length $q_1 + p_2-1 -p_1 + n-s+1 -q_2 = (n-s+1) + (p_2-p_1) -(q_2-q_1) -1 \geq n-s+1$ and the index of $y_{p_1+1}$ in $w$ is larger than the index of $\sigma_{q_1+1}$ (which is the index of $y_{p_1}$ in $w'$). We thus reach a contradiction, showing that we have found the inverse map of $\Phi$ is given by $\Psi(w,a)=w'$ and, in particular, that $\Phi$ is injective. 

We have also shown that the image of $\Phi$ consists exactly of these permutations, which do not have an increasing subsequence of length $n-s+2$ within their first $a$ elements.  Therefore the coimage of $\Phi$ is the set of permutations in $C_{n,s-1}$ with $n-s+2$ increasing subsequence within its first $a$ elements, so the ones in $C_{n,s-1,b}$ for $n-s+2 \leq b \leq a$. 
\end{proof}

We can now proceed to the proof of Theorem \ref{garsia}. We have that $C_{n,k}\setminus \R_{n,k}$ is exactly the set of permutations in $C_{n,k}$ with some increasing subsequence of length $n-k+1$, hence
$$ C_{n,k} \setminus \R_{n,k} = \bigcup_{n-k+1\leq a_1 \leq n} C_{n,k,a_1} .$$
On the other hand, applying the lemma we have that 
{\allowdisplaybreaks 
\begin{align}\label{permsets}
 \bigcup_{n-k+1\leq a_1 \leq n}& C_{n,k,a_1}  \simeq \bigcup_{n-k+1\leq a_1 \leq n} \Phi(C_{n,k,a_1}) \\
 &=\bigcup_{n-k+1\leq a_1 \leq n} \left( C_{n,k-1}\times a_1 \setminus \bigcup_{n-k+2 \leq a_2 \leq a_1} C_{n,k-1,a_2} \right) \notag \\
 &= C_{n,k-1}\times \binom{[k]}{1} \setminus \left( \bigcup_{n-k+2\leq a_2 \leq a_1 \leq n} C_{n,k-1,a_2}\times a_1 \right) \notag \\
 &\simeq C_{n,k-1}\times \binom{[k]}{1} \setminus \left( C_{n,k-1} \times \binom{[k]}{2} \setminus \left(\bigcup_{n-k+3  \leq a_3 \leq a_2 \leq a_1 \leq n} \mspace{-10mu} C_{n,k-2,a_3}\times(a_2,a_1)\right) \right)
  \notag \\
=\cdots &= C_{n,k-1}\times \binom{[k]}{1} \setminus \left(C_{n,k-2}\binom{[k]}{2} \setminus \dots \setminus \left( C_{n,k-r}\times \binom{[k]}{r} \setminus \dots\right) \dots \right) , \notag
 \end{align}
 }
 where $\simeq$ denotes the equivalence under $\Phi$ and  $\binom{[k]}{r}$ represent the $r-$tuples $(a_r,\ldots,a_1)$ where $n-k+r \leq a_r \leq a_{r-1}\leq \cdots \leq a_1\leq n$. Since $\# C_{n,k-r} \times \binom{[k]}{r} = \binom{n}{k-r}(k-r)! \binom{k}{r}$, Theorem \ref{garsia} follows. 
 
As for the $q-$analogue, Theorem \ref{garsiaqanalogue}, it follows immediately from the set equalities \eqref{permsets} and Lemma \ref{q-inv} \, once we realize that the map $\Phi$ does not change the major index of the inverse permutation, as shown in the following small lemma. 
\begin{lemma}
Let $D(w)=\{i+1 \text{ before } i \text{ in }w\}$. Then $D(w) = D(\Phi(w))$ and thus $\maj(w^{-1}) = \sum_{i \in D(w)} i$ remains the same after applying $\Phi$. 
\end{lemma}
\begin{proof}
To see this, notice that  $i$ and $i+1$ could hypothetically change their relative order after applying $\Phi$ only if exactly one of them is in the LLI-$(n-s+1)$ sequence of $w$, denote this sequence by $\sigma = w_1,\dots,w_{i_{n-s+1}}$. 

Let $w_p=i$ and $w_q=i+1$. We need to check only the cases when $p=i_r$ and $i_{r-1}<q<i_r$ or $q =i_r$ and $i_{r-1}<p<i_r$, since otherwise $i$ and $i+1$ preserve their relative order after shifting $\sigma$ one step forward by applying $\Phi$. In either case, we see that the sequence $w_1,\ldots,w_{i_{r-1}},w_p\text{ (or } w_q), w_{i_{r+1}}, \ldots, w_{i_{n-s+1}}$ is increasing of length $n-s+1$ in $w$ and has lexicographically smaller indices than $\sigma$, violating the LLI property. Thus these cases are not possible and the relative order of $i$ and $i+1$ is preserved, so $D(w) = D(\Phi(w))$.
\end{proof}

We now have that the equalities and equivalences in \eqref{permsets} are equalities on the sets $D(w)$ and so preserve the $\maj(w^{-1})$ statistic, leading directly to Theorem \ref{garsiaqanalogue}.

\bibliographystyle{abbrvnat}
\bibliography{sample}
\label{sec:biblio}

\end{document}